\documentclass[11pt]{amsart}
\usepackage{leqno,color,psfrag,epsfig,url,hhline%,showkeys
}
\usepackage{hyperref}
\usepackage[dvipsnames]{xcolor}
\numberwithin{figure}{section}
\numberwithin{table}{section}
\newtheorem{theorem}{Theorem}[section]

\newtheorem{lemma}[theorem]{Lemma}
\newtheorem{prop}[theorem]{Proposition}

\theoremstyle{definition}

\newtheorem{example}[theorem]{Example}

\newtheorem{cor}[theorem]{Corollary}

\theoremstyle{remark}
\newtheorem{remark}[theorem]{Remark}

\numberwithin{equation}{section}
\newfont{\tap}{tap scaled 650}

\def \H{{\mathbb H}}

\def \R{{\mathbb R}}

\def \Z{{\mathbb Z}}

\def \Q{{\mathbb Q}}

\def \[{[ }
\def \]{] }

\newcommand{\im}{{\mathrm{Im}\,}}
\definecolor{dgreen}{rgb}{0,0.5,0}        
\definecolor{dred}{rgb}{0.5,0,0}

\begin{document}

\title{$SL_2$-tilings with translational symmetry}
\author{V\'eronique Bazier-Matte}
\author{Marie-Anne Bourgie}
\address{D\'epartement de math\'ematiques et de statistique, Universit\'e Laval, Qu\'ebec (Qu\'ebec), G1V 0A6, Canada}
\email{veronique.bazier-matte.1@ulaval.ca, marie-anne.bourgie.1@ulaval.ca}
  \author{Anna Felikson}
\author{Pavel Tumarkin}
\address{Department of Mathematical Sciences, Durham University, Mathematical Sciences \& Computer Science Building, Upper Mountjoy Campus, Stockton Road, Durham, DH1 3LE, UK}
\email{anna.felikson@durham.ac.uk, pavel.tumarkin@durham.ac.uk}

\begin{abstract}
An $SL_2$-tiling is a bi-infinite matrix  in which all adjacent
$2 \times 2$ minors are equal to~$1$. Positive integral $SL_2$-tilings were introduced by Assem, Reutenauer and Smith as  generalisations of classical  Conway--Coxeter frieze patterns. We show that positive integral $SL_2$-tilings with translational symmetry are in bijection with triangulations of annuli. We use this correspondence  to study the properties of periodic positive integral $SL_2$-tilings.

\end{abstract}

 \maketitle

\thispagestyle{empty}
 
 \section{Introduction}

 Frieze patterns were introduced by Coxeter~\cite{C} and studied by Conway and Coxeter~\cite{CC} in 1970s. These are 
 grids of numbers arranged as in Fig.~\ref{fig-fr}, with finitely many rows, a row of $0$'s and $1$'s both at the top and the bottom, and such that  the entries in every small diamond  $a$\put(3,5){$b$}\put(3,-6){$c$} \quad  $d$  satisfy the  {\it diamond  rule}: $ad - bc = 1$.

\begin{figure}[!h]
\epsfig{file=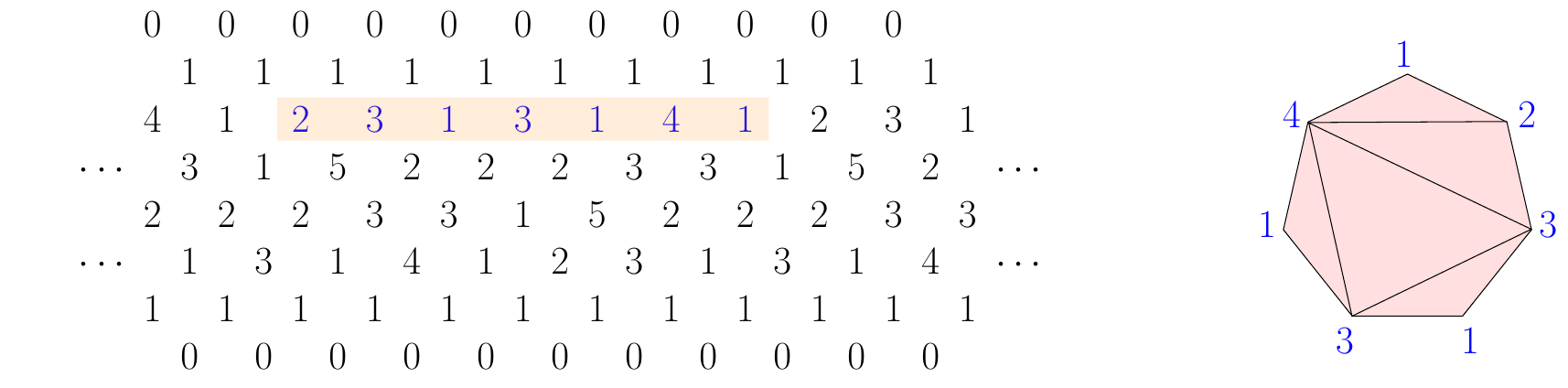,width=0.95\linewidth}  
\caption{(a) Conway--Coxeter's frieze pattern; (b) the corresponding triangulated polygon. }
\label{fig-fr}
\end{figure}

Conway and Coxeter~\cite{CC} showed that positive integral frieze patterns are in bijection with triangulated polygons, and based on that  derived a number of properties of positive integral frieze patterns.
Most importantly, every Conway--Coxeter's frieze pattern is periodic (i.e., invariant under a finite shift along its row), and is completely determined by its first non-trivial row, whose period  (highlighted in Fig.~\ref{fig-fr}) is called a {\it quiddity sequence} and consists of the numbers of triangles incident to the vertices in the triangulated polygon.

It turned out later that Conway--Coxeter's frieze patterns are closely related to cluster algebras, combinatorics, geometry and representation theory. Numerous generalizations of them were considered and were shown to exhibit various subsets of  properties of classical frieze patterns. 
One of the most investigated generalizations is called {$SL_2$-tilings} and was introduced by Assem, Reutenauer and Smith in~\cite{ARS}:

An {\it $SL_2$-tiling} is a bi-infinite matrix $(u_{i,j})$, $i, j\in \Z$, whose all adjacent $2\times 2$ minors are equal to $1$. An $SL_2$-tiling is \emph{integral} if  all $u_{i,j}\in\Z$, and \emph{positive} if all $u_{i,j}>0$.

\medskip
\noindent
In this paper, we consider {\it positive integral $SL_2$-tilings with translational symmetries}, i.e. $SL_2$-tilings for which
$$u_{i,j}=u_{i+m,j+n}$$ 
for some $m,n\in \Z$ and all $i,j\in \Z$ (we call such $SL_2$-tilings \emph{$(m,n)$-periodic}). See Fig.~\ref{u-sl}(a) for an example of a $(3,2)$-periodic $SL_2$-tiling. Notice that we use the indexing which is slightly different from the usual one. 

 The main result of this paper is  the following counterpart of Conway--Coxeter's theorem:

\begin{theorem}[Theorem~\ref{thm-main}, Corollary~\ref{cor-bijection}]
There is a one-to-one correspondence between positive integral  $SL_2$-tilings with translational symmetry and triangulations of annuli.  

\end{theorem}

More precisely, $(m,n)$-periodic $SL_2$-tilings correspond to triangulations of annuli with $m$ and $n$ marked points on its boundary components.

This can be restated in terms of friezes on surfaces: $SL_2$-tilings with translational symmetry are in bijection with friezes on annuli. More precisely, given a frieze $F$ on an annulus, the corresponding $SL_2$-tiling is filled with values of $F$ on all bridging arcs of the annulus  (see Section~\ref{fr-on-surf} for the definition and Section~\ref{sec mn} for details). 

The proof is partially based on the direct construction (see Section~\ref{sec mn}), and partially on the Short's characterisation~\cite{Sh} of tame $SL_2$-tilings in terms of a pair of paths on the Farey graph (see Section~\ref{sec-sl2}).

It was shown by Holm and Jorgensen~\cite{HJ} that  positive integral $SL_2$-tilings {\it with enough 1's} are in bijection with triangulations of infinite strips. However, we are not able to use this result directly, as \emph{a priori} it is not clear whether an $SL_2$-tiling with translational symmetry is required to contain any $1$'s at all. Nevertheless, after the present paper was written, it was noticed to the authors by Peter Jorgensen that the main result can be also deduced from a combination of results of~\cite{HJ} and~\cite{BHJ}, see Remark~\ref{J}.

\medskip
Notice that triangulations of annuli were previously associated with periodic {\it infinite frieze patterns} (infinitely tall tables starting with a row of $0$'s and a row of $1$'s and satisfying the diamond rule): given a triangulated annulus, a periodic infinite frieze pattern can be read off from each of the two boundary components of the annulus.
It was shown by Baur, Parsons and Tschabold~\cite{BPT} that every periodic infinite frieze pattern arises from some triangulated annulus. 
The two infinite frieze patterns and the $SL_2$-tiling arising from the same annulus are tightly connected:

\begin{prop}[Propositions~\ref{prop-quid},~\ref{prop-ets}]
Given an $SL_2$-tiling with translational symmetry, there are explicit formulae to compute the entries of the two infinite frieze patterns arising from the same annulus, and to compute the growth of these infinite frieze patterns (in terms of elements of the $SL_2$-tiling). 

\end{prop}

Given a finite piece of table filled by positive integers, it may not be immediately clear whether this piece can be extended to an $(m,n)$-periodic $SL_2$-tiling. We address this question in terms of a zig-zag of positive integers (or more precisely, lattice paths between positions $(i,j)$ and $(i+m,j+n)$). Notice that such a zig-zag always extends uniquely to an $(m,n)$-periodic table with unit $2\times 2$ determinants, however, {\it a priori} it is not clear whether all entries of the table will be integers. 

\begin{prop}[Proposition~\ref{zig-zag}]  
Given a  lattice path between positions $(i,j)$ and $(i+m,j+n)$ such that $u_{i',j'}$ is positive integer for all $(i',j')$ in the lattice path, and $u_{i,j}=u_{i+m,j+n}$, there are $m+n$ explicit conditions to check: if all these conditions hold then there is a unique way to extend the given collection to a positive integral $SL_2$-tiling, otherwise, there is none.  
  
\end{prop}

Finally, we investigate the appearance of $1$'s in an $SL_2$-tiling with translational symmetries:

\begin{prop}[Proposition~\ref{prop-ears}]
  Given an $(m,n)$-periodic $SL_2$-tiling, either there exists a zig-zag of 1's between $(i_0,j_0)$ and $(i_0+m,j_0+n)$, or there exists $i_0$ (or $j_0$) such that the row $u_{i_0,j}$ is the sum of rows 
  $u_{i_0-1,j}$ and $u_{i_0+1,j}$ (respectively, the column  $u_{i,j_0}$ is the sum of the columns 
  $u_{i,j_0-1}$ and $u_{i,j_0+1}$).  In  the latter case, removing all rows $u_{i_0 +mt,j}$ (resp., all columns $u_{i,j_0 +nt}$) for all $t \in \Z$ leads
to an $(m-1,n)$-periodic $SL_2$-tiling (resp., $(m, n-1)$-periodic). 

\end{prop}  

This result agrees with the results in~\cite{HJ}.

We note that twice periodic $SL_2$-tilings containing non-positive integers were previously studied by Morier-Genoud, Ovsienko, and Tabachnikov~\cite{MGOT}: it was proved that (under some assumptions on the signs of entries) such $SL_2$-tilings arise from a pair of closed paths on the Farey graph.

\bigskip
\noindent
The paper is organised as follows. In Section~\ref{sec-prelim} we recall some necessary notions such as friezes from surfaces (more specifically, from annuli), Farey graph and its connection to $SL_2$-tilings.
In Section~\ref{sec n}, we introduce {\it bi-infinite quasi-friezes (BiQF)}: it is a bi-infinite modification of Conway--Coxeter frieze pattern satisfying a quasi-diamond rule $ad-bc=-1$. We study $n$-periodic  BiQFs (i.e. BiQFs with horizontal translational symmetry $u_{i,j}=u_{i+n,j+n}$, see Fig.~\ref{fig-u}) and conclude that they are in bijection with triangulations of annuli $A_{n,n}$ with $n$ marked points at every boundary component. In Section~\ref{sec mn}, we consider BiQFs with non-horizontal translational symmetry $u_{i,j}=u_{i+m,j+n}$ and show that such BiQFs are in bijection with triangulations of annuli $A_{m,n}$ (with $m$ marked points on one boundary component and $n$ on another).  In Section~\ref{sec-sl2-transl} we reformulate the above results for $SL_2$-tilings (as a BiQF rotated by $45^\circ$ counterclockwise is an $SL_2$-tiling) and show that $SL_2$-tilings admit no translational symmetries other than the ones considered above. We use results of~\cite{Sh} to prove Theorem~\ref{thm-main}. In Section~\ref{sec-properties} we explore the properties of $(m,n)$-periodic $SL_k$-tilings.

\subsection*{Acknowledgements}
We are grateful to Andrei Zabolotskii for useful discussions, and to Peter Jorgensen for helpful suggestions. The first author was supported by Fonds de recherche du Qu\'ebec - Nature et technologie, grant  327749 "Research Support for New Academics".

\section{Preliminaries}
\label{sec-prelim}

In this section we recall some results concerning friezes on annuli, $SL_2$-tilings and Farey graph.

\subsection{Friezes on annuli}
\label{fr-on-surf}
Consider a marked surface  $(S,M)$ where $S$ is a surface and $M$ is a collection of marked points on it. In this paper, $S$ will be an annulus,  and $M$ will be a collection of $m$ marked points  $P_1,\dots,P_m$ on one boundary component and $n$ marked points $Q_1,\dots,Q_n$ on the other boundary component of the annulus, $m,n>0$. We denote such a marked annulus by $A_{m,n}$.

By an arc on $(S,M)$ we mean a non-self-intersecting curve in $S$ with endpoints in $M$, considered up to isotopy. On the annulus $A_{m,n}$, the arcs with two endpoints on different boundary components will be called {\it bridging}, and all other arcs will be called {\it peripheral}.
A {\it triangulation} of $(S,M)$ is  a collection of arcs subdividing $S$ into triangles with vertices in $M$.
By a quadrilateral $ABCD$ on $(S,M)$ we mean a collection of disjoint arcs $AB,BC,CD,DA$ cutting out a disc $ABCD$ from $S$.

Given a marked surface $(S,M)$, an {\it integral positive frieze} $F$ on $(S,M)$ (or simply a
{\it frieze} for short) is an assignment of positive integers to all arcs on $S$, such that $F(\delta)=1$ for every boundary arc $\delta$, and for every quadrilateral $ABCD$ on $(S,M)$ the \emph{Ptolemy relation} holds:
$$ F(AC)\cdot F(BD)=F(AB)\cdot F(CD) + F(BC)\cdot F(AD).
$$

A frieze $F$ on $(S,M)$ is called {\it unitary} if there exists a triangulation $T$ of $(S,M)$ such that  $F(\gamma)=1$ for all $\gamma\in T$.
In this case, we will say that this triangulation $T$ is  the {\it unitary triangulation} for $F$ (such a triangulation is unique, and moreover, in any positive integral frieze on a marked surface, the Ptolemy relation implies that two arcs labeled 1 cannot intersect each other).

Conway and Coxeter~\cite{CC} showed that friezes on polygons are always unitary. Furthermore, it was shown by Gunawan and Schiffler~\cite{GS} that every positive integral frieze on an annulus is unitary (and hence, there is a bijection between positive integral friezes on annuli and triangulations of annuli). 

\begin{remark}(Friezes on surfaces and cluster algebras)
  \label{rem-cl}
  Friezes on marked surface $(S,M)$ can be understood as homomorphisms of the corresponding cluster algebra $\mathcal A(S,M)$ to $\R$. More precisely, according to~\cite{FST,FT} cluster variables of $\mathcal A(S,M)$ correspond to arcs on $(S,M)$, where clusters correspond to triangulations. Now, the values of the frieze are precisely the images of cluster variables.  In particular, positive integral friezes are homomorphism with images of all cluster variables being positive integers.   It follows from~\cite[Theorem 3.1]{P} that if values of a frieze on a triangulation are all positive, then the whole frieze is positive.
  
\end{remark}
  
\subsection{$SL_2$-tilings}
An $SL_2$-tiling is a bi-infinite matrix  such that all adjacent $2 \times 2$ minors are equal to  1.
An $SL_2$-tiling is {\it tame} if every adjacent $3\times 3$ determinant vanishes.
We will denote the elements of an $SL_2$-tiling $u_{i,j}$ and index them as in Fig.~\ref{u-sl}.

In this paper we will discuss {\it positive integral} $SL_2$-tilings.
A positive integral $SL_2$-tilings is always tame (this immediately follows from Dodgson’s condensation rule, see~e.g.~\cite[Section 3]{BR}).

\begin{figure}[!h]
  \epsfig{file=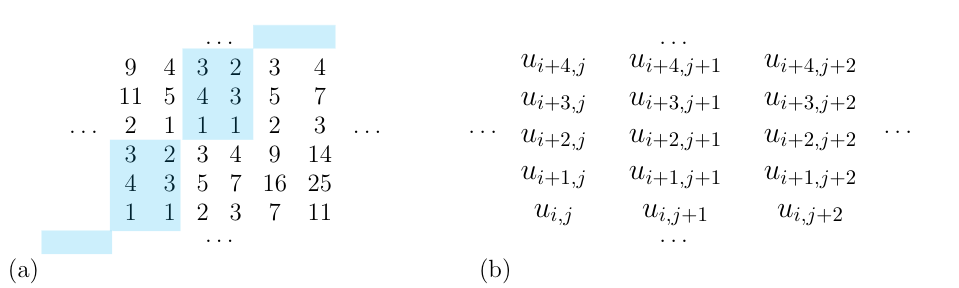,width=0.8\linewidth}  
\caption{(a) An example of a $(3,2)$-periodic $SL_2$-tiling; (b) the numeration of entries in $SL_2$ tiling.}
\label{u-sl}
\end{figure}

\begin{remark}[Warning on the indexing]
\label{warning}  
Notice, that in our indexing the numeration of rows grows bottom to top (unlike many other sources), this is convenient for ensuring the symmetry between rows and columns in further considerations.
This should be taken into account when comparing to results of other papers
% ~\cite{HJ}
or using the formulae from Section~\ref{sec-properties}.

\end{remark}

\subsection{$SL_2$-tilings and paths on the Farey graph}
\label{sec-sl2}

The Farey graph $\mathcal F$ is the graph with vertices at $\Q\cup \{\infty\}$, where the
two vertices given by reduced fractions $p/q$ and $r/s$ are connected by an edge whenever $|ps-qr| = 1$.

It is convenient to depict the Farey graph on the upper half-plane $\im  z>0$  (connecting the points  $p/q$ and $r/s$ on $\R$ by a half-circle). The arcs of $\mathcal F$  triangulate the upper half-plane (into curvilinear triangles that can be interpreted as ideal triangles in hyperbolic geometry). Alternatively, one can map the same diagram to the disc, see Fig.~\ref{fig-farey}.
The group $SL_2(\Z)$ preserves $\mathcal F$ and acts transitively on the vertices of $\mathcal F$. 
In particular, every vertex of the Farey graph is incident to infinitely many edges.

\begin{figure}[!h]
  \epsfig{file=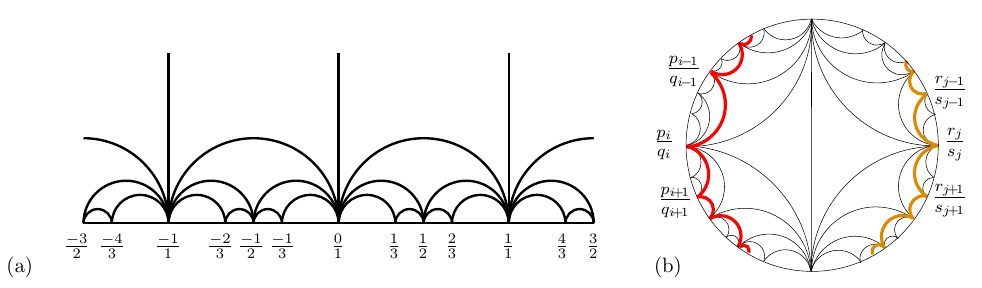,width=0.95\linewidth}  
\caption{(a) Farey graph in the upper half-plane; (b) same in the disc, with two
paths highlighted.}
\label{fig-farey}
\end{figure}

Denote $\det(p/q,r/s)=\begin{vmatrix} p&r\\q&s \end{vmatrix}$.
These determinants are the key for constructing $SL_2$-tilings from paths on Farey graph (see Theorem~\ref{thm-sh})
 due to the following  version of the Ptolemy relation  shown in~\cite{MGO}.
 
\begin{prop}[~\cite{MGO}]
 \label{det} 
  Given four reduced fractions $a_1/a_2, b_1/b_2,c_1/c_2,d_1/d_2\in \Q$ lying on the boundary of the disc in this cyclic order, the following holds
$$
\begin{vmatrix} a_1&c_1\\a_2&c_2 \end{vmatrix}\cdot\begin{vmatrix} b_1&d_1\\b_2&d_2 \end{vmatrix}=
\begin{vmatrix} a_1&b_1\\a_2&b_2 \end{vmatrix}\cdot\begin{vmatrix} c_1&d_1\\c_2&d_2 \end{vmatrix}+
\begin{vmatrix} a_1&d_1\\a_2&d_2 \end{vmatrix}\cdot\begin{vmatrix} b_1&c_1\\b_2&c_2 \end{vmatrix}
$$

\end{prop}

A bi-infinite sequence of reduced fractions $(p_i /q_i )_{i\in \Z}$ represents a \emph{path} in the Farey graph $\mathcal F$ if
$p_i q_{i+1}- q_i p_{i+1} = 1$ for all $i\in \Z$.
It was shown by Short~\cite{Sh} that every tame $SL_2$-tilings can be obtained from a pair of paths on $\mathcal F$.

\begin{theorem}[\cite{Sh}]
\label{thm-sh}  
\begin{itemize}
\item[(a)]  
There is a bijection between tame $SL_2$-tilings (modulo multiplication by $-1$) and pairs of paths  $({p_i/q_i},{r_j/s_j})$ on $\mathcal F$ (modulo simultaneous action of $SL_2(\Z)$). This bijection is given by
$$
u_{i,j}= \begin{vmatrix} p_i & r_j \\ q_i & s_j \end{vmatrix}.
$$

\item[(b)]
In particular, positive $SL_2$-tilings are in bijection with pairs of paths $((\frac{p_i}{q_i}),(\frac{r_j}{s_j}))$, where the vertices of the paths lie on the boundary of the disc clockwise in the following cyclic order: $...,\frac{p_{i+1}}{q_{i+1}}, \frac{p_i}{q_i},\frac{p_{i-1}}{q_{i-1}},\dots, \frac{r_{j-1}}{s_{j-1}}, \frac{r_j}{s_j}, \frac{r_{j+1}}{s_{j+1}}...\,$.

\end{itemize}  
\end{theorem}

\begin{remark}
\label{rem-paths}  
Notice that in in~\cite{Sh} both paths follow the boundary of the disc clockwise, but we changed the direction of the path $(p_i/q_i)$ to adapt the result to our notation as the numeration of rows in our notation is bottom to top, see Remark~\ref{warning}.
  
\end{remark}

\section{Bi-infinite periodic quasi-friezes  (BiQF)}
\label{sec n}

By a {\it bi-infinite quasi-frieze} (BiQF) we mean
a bi-infinite grid of real numbers $u_{i,j}$, where $i,j\in\Z$, the entries are arranged as in Fig.~\ref{fig-u}(b),
and the entries in every small diamond  $a$\put(3,5){$b$}\put(3,-6){$c$} \quad  $d$  satisfy the  {\it quasi-diamond  rule}: $ad - bc = -1$.

\begin{figure}[!h]
  \epsfig{file=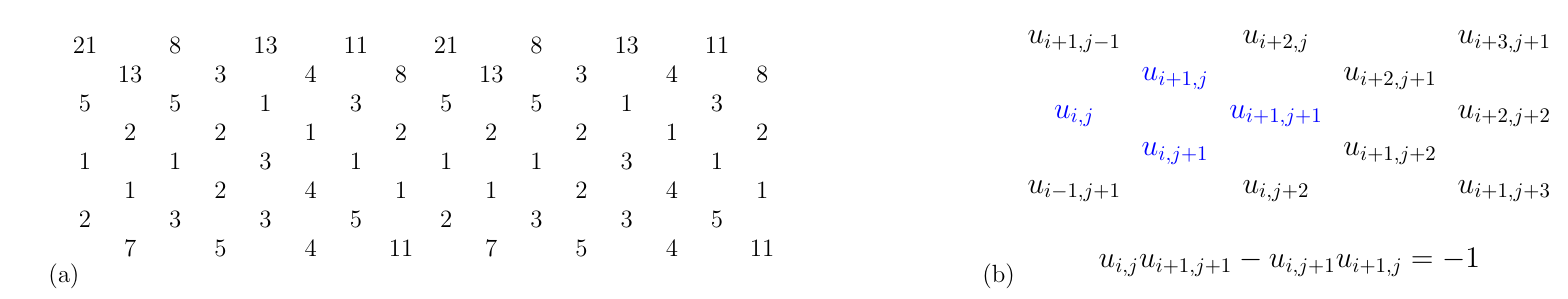,width=0.99\linewidth}  
\caption{Example of a $4$-periodic BiQF, indexing of its entries and the quasi-diamond rule. }
\label{fig-u}
\end{figure}

We will consider  {\em positive integral} BiQFs, so from now on we assume $u_{i,j}\in \Z$ and $u_{i,j}>0$.  

We say that a BiQF is $n$-{\em periodic} (with period $n$) if $u_{i+n,j+n}=u_{i,j}$ for all $i,j$.

\begin{example}
\label{ex-n}  
  Consider an annulus $A_{n,n}$ with $n$ marked points at each of the boundary components. Label the marked points on the two boundary components $P_1,\dots,P_n$ and $Q_1,\dots,Q_n$ respectively, so the numbering on each component grows counterclockwise as in Fig.~\ref{fig-A22} 
  (here $P_iP_{i+1}$ and  $Q_jQ_{j+1}$ are boundary arcs, and $P_{n+1}$ and $Q_{n+1}$ are identified with $P_1$ and $Q_1$ respectively).

    Consider a frieze $F$ on the annulus, and let $T$ be the triangulation of the annulus containing arcs $P_iQ_i$ and $P_iQ_{i+1}$ for all $i=1,\dots,n$. We also consider the universal cover of  $A_{n,n}$  which is an infinite strip with marked points $\hat P_i$ and $\hat Q_j$, $i,j\in\Z$, and the lift $\hat T$ of triangulation $T$ to it. Denote by $\hat F$ the lift of $F$ to the universal cover (i.e. $\hat F(\hat \gamma):=F(\gamma)$ for every lift $\hat \gamma$ of every arc $\gamma\in A_{n,n}$).
  Denote
  \begin{equation}
   \label{u=F} 
    u_{i,j}=\hat F(\hat P_i\hat Q_j),\qquad i,j\in\Z,
\end{equation}
    and arrange the numbers $u_{i,j}$ in a grid as in Fig.~\ref{fig-u}(b). Every small diamond then corresponds to a quadrilateral
    $P_iP_{i+1}Q_{j+1}Q_j$ with two opposite sides on two boundary components. Since $F(P_iP_{i+1})=F(Q_{j+1}Q_j)=1$, the Ptolemy relations in the quadrilaterals $P_iP_{i+1}Q_{j+1}Q_j$ imply quasi-diamond conditions on $u_{i,j}$, so we obtain an $n$-periodic BiQF.
    
In particular, the values of $F$ on the arcs of the triangulation $T$ give two consecutive rows of the BiQF (one row consists of $F(P_iQ_i)$, $i=1,\dots,n$, and another of $F(P_iQ_{i+1})$).
See Fig~\ref{fig-A22} for an example of a $3$-periodic BiQF arising from a frieze on the annulus $A_{3,3}$. 

\end{example}

\begin{figure}[!h]
  \epsfig{file=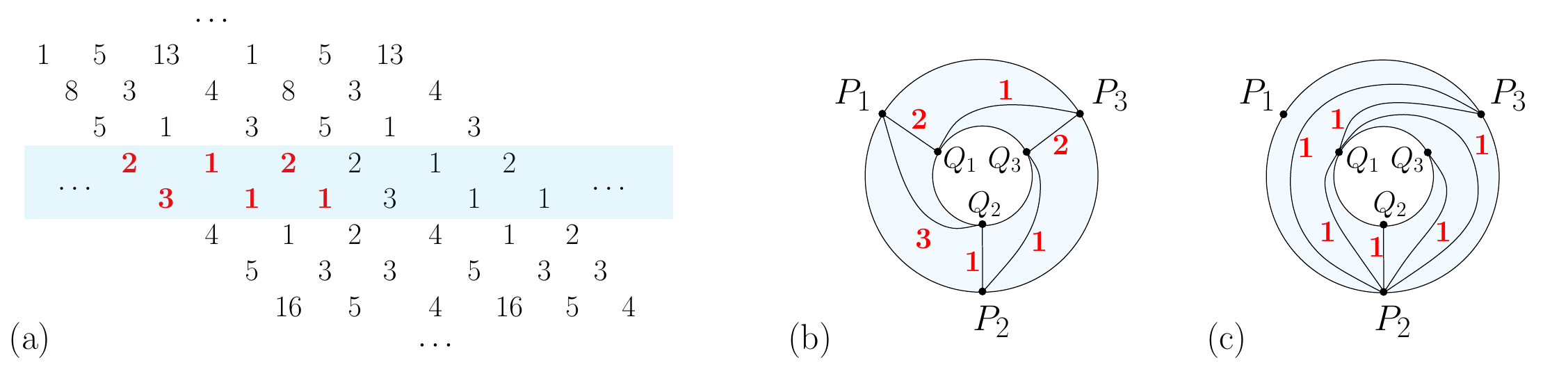,width=0.99\linewidth}  
\caption{Example of a $3$-periodic BiQF arising from a triangulated annulus $A_{3,3}$  (see Example~\ref{ex-n}): (a) BiQF; (b) triangulation $T$ and values $F(\gamma)$ for all $\gamma\in T$; (c) the corresponding unitary triangulation. }
\label{fig-A22}

\end{figure}

\begin{remark}
  \begin{itemize}
\item[(a)] Notice that all entries of the BiQF constructed in  Example~\ref{ex-n} correspond to the bridging arcs of the annulus (unlike the entries of infinite friezes in~\cite{BPT} which correspond to peripheral arcs).

\item[(b)] Consider two consecutive rows of a BiQF from Example~\ref{ex-n}, $u_{i,i}$ and $u_{i,i+1}$,  $i\in\Z$. Every pair of adjacent entries in these rows (i.e., $(u_{i,i},u_{i,i+1})$ or  $(u_{i,i+1},u_{i+1,i+1})$)  corresponds to a triangle of the triangulation $T$ (where the third side is a boundary arc).
  
\end{itemize}
\end{remark}

The following example is a counterpart of~\cite[(34)]{CC}. 

\begin{example}
\label{ex-fib}  
A $1$-periodic BiQF can be constructed from the annulus $A_{1,1}$, it has two consecutive rows of $1$'s (notice that $A_{1,1}$ has a unique triangulation up to Dehn twists). It is easy to check that rows are filled with even terms of the Fibonacci sequence, see Fig.~\ref{fig-fib}(a).

\end{example}   

\begin{figure}[!h]
  \epsfig{file=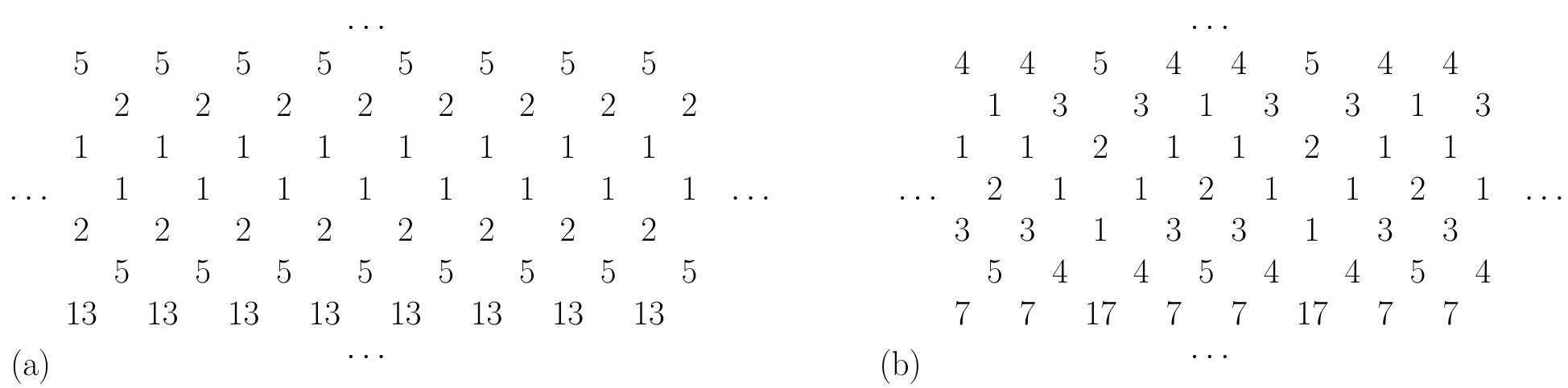,width=0.98\linewidth}  
\caption{(a) A 1-periodic BiQF arising from $A_{1,1}$ is filled with Fibonacci numbers, see Example~\ref{ex-fib}; (b) a 3-periodic BiQF invariant under a glide reflection, see Remark~\ref{rem-mob}.}
\label{fig-fib}

\end{figure}

Example~\ref{ex-n} shows for every $n$ one can construct an $n$-periodic BiQF by using a frieze on an annulus $A_{n,n}$. Our next aim is to prove that every  $n$-periodic BiQF can be obtained in this way.
To prove this, we will need the following technical lemma (which we formulate in greater generality).

\begin{lemma}
\label{periph}  
Let $A_{m,n}$ be an annulus and $F$ be a frieze on  $A_{m,n}$ (here we do {\em not} assume $F$ to be integral or positive).
 If $F(\gamma)\in \Z$ for every bridging arc $\gamma$ on  $A_{m,n}$,   then $F$ is integral, i.e. $F(\gamma)\in \Z$ for all arcs on  $A_{m,n}$.
\end{lemma}

\begin{proof}
It is sufficient to show that $F(\gamma)$ is integer for every peripheral arc $\gamma$.
Let  $P_iP_k$ be a peripheral arc. Consider a  quadrilateral $P_iP_kQ_2Q_1$, where $Q_1Q_2$ is a boundary segment ($Q_1$ may coincide with $Q_2$ if the corresponding boundary component contains one marked point only), see Fig~\ref{fig-periph}(a) for the picture on the universal cover. By Ptolemy relation in this quadrilateral we have $$F(P_iP_k)\cdot 1 = F(P_iP_k)\cdot F(Q_1Q_2)=  F(P_iQ_2 )\cdot F(P_kQ_1) - F(P_kQ_2 )\cdot F(Q_1P_i)\in \Z. $$

\end{proof}  

\begin{figure}[!h]
  \epsfig{file=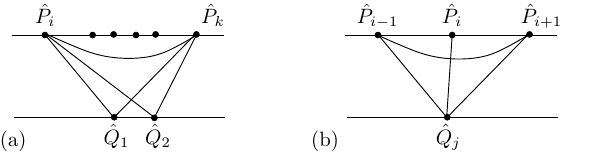,width=0.67\linewidth}  
\caption{Computing the values of the frieze on peripheral arcs from the values on bridging arcs. }
\label{fig-periph}

\end{figure}

\begin{theorem}
\label{thm-n}  
There is a bijection between  positive integral $n$-periodic BiQFs and positive integral friezes on $A_{n,n}$.

\end{theorem}

\begin{proof}
  Given a frieze on the annulus  $A_{n,n}$, we consider a triangulation as in Example~\ref{ex-n} and construct an $n$-periodic BiQF by applying Equation~(\ref{u=F}).
  We will now provide the inverse construction.
  
  Consider a positive integral $n$-periodic BiQF with elements $u_{i,j}$, take the elements in two rows $u_{i,i}$ and $u_{i,i+1}$, where $i\in\Z$.
  Consider an annulus $A_{n,n}$ with a triangulation $T$ as in Example~\ref{ex-n}. Define $F(P_iQ_i)=u_{i,i}$ and $F(P_iQ_{i+1})=u_{i,i+1}$. Define also $F(P_iP_{i+1})=F(Q_iQ_{i+1})=1$. Applying Ptolemy relations starting from the triangulation $T$ and flipping the arcs, we will obtain some value $F(\gamma)$ for every arc $\gamma$ on $A_{n,n}$ (and the result will not depend on  the sequence of flips bringing to a given arc, see Remark~\ref{rem-cl}).
  
    Notice that  due to the quasi-diamond rule in BiQF we will have $F(P_iQ_j)=u_{i,j}$. So, we will get integers on all bridging arcs of the annulus. By Lemma~\ref{periph} we conclude that $F(\gamma)\in \Z$  also for all peripheral arcs on the annulus, which (in view of Remark~\ref{rem-cl}) implies that $F$ is a positive integral frieze on   $A_{n,n}$.

Observe that starting from a different pair of rows of a BiQF leads to the same frieze, so the map is well defined. The map is clearly inverse to the one described in Example~\ref{ex-n}, which implies that these maps are bijections.

\end{proof}

\begin{cor}
There is a bijection between positive integral $n$-periodic BiQFs and triangulations of $A_{n,n}$.
  
\end{cor}

\begin{proof}
The statement follows from Theorem~\ref{thm-n} combined with the result of~\cite{GS} stating that every frieze on an annulus is unitary (and hence friezes on the annulus are in bijection with triangulations of the annulus).

\end{proof}

\begin{remark}
Positive integral $n$-periodic BiQFs considered up to any shift $(i,j)\to (i+k,j+l)$ of the whole grid are in bijection with triangulations of $A_{n,n}$ considered up to cyclic relabelling of the points on the two boundaries.

\end{remark}

\begin{remark}
\label{rem-mob}
  If a BiQF in addition to $n$-periodicity admits a glide symmetry  (see Fig.~\ref{fig-fib} for  examples) then the frieze on corresponding annulus admits an additional symmetry, which can be understood as a symmetry of the surface. Taking a quotient we obtain a frieze on a M\"obius band $M_n$.
More generally, positive integral $n$-periodic BiQFs with glide symmetry are in bijection with friezes on $M_n$ (the entries of the BiQF correspond to arcs and self-intersecting curves on $M_n$ passing through the cross-cap), and thus are naturally linked with quasi-cluster algebras introduced in~\cite{DP}.

\end{remark}

\section{$(m,n)$-periodic bi-infinite quasi-friezes}
\label{sec mn}
In this section we consider BiQFs with translational symmetry  $u(i,j)=u(i+m,j+n)$, we will call it $(m,n)$-periodic BiQF. See Fig.~\ref{fig-mn} for an example of a $(3,2)$-periodic BiQF.  
In particular, $n$-periodic BiQFs from Section~\ref{sec n} are $(n,n)$-periodic BiQFs.

\begin{figure}[!h]
 \epsfig{file=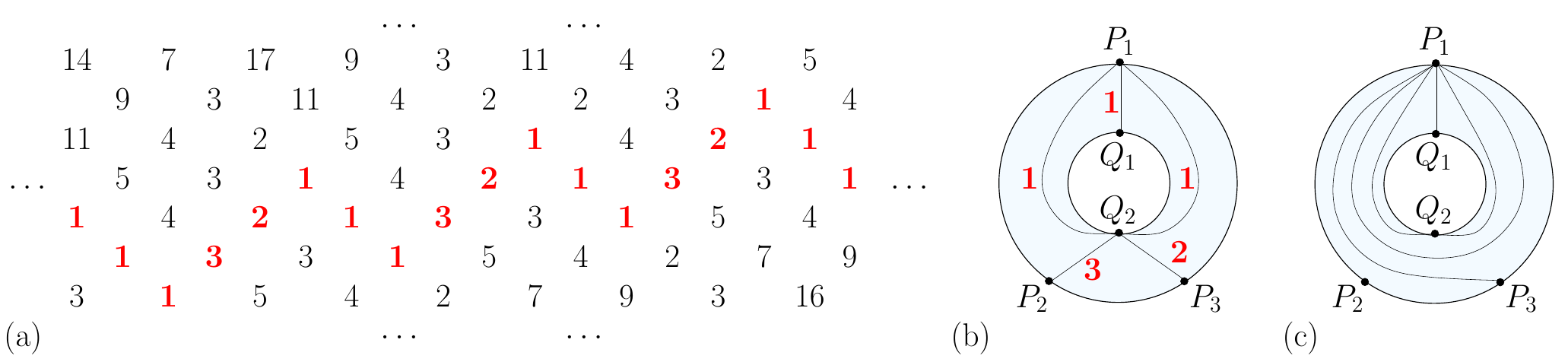,width=0.95\linewidth}  
\caption{Example of a $(3,2)$-periodic BiQF (see Example~\ref{ex-mn}): (a) BiQF; (b) the corresponding frieze on $A_{3,2}$ (the triangulation corresponding to red entries in (a)); (c) the corresponding unitary triangulation.}
\label{fig-mn}
\end{figure}

\begin{example}
\label{ex-mn}  
  Consider an annulus $A_{m,n}$  with $m$ marked points $P_1,\dots,P_m$ on one boundary component and $n$ marked points $Q_1, \dots, Q_n$ on the other. Consider also an infinite strip giving the universal cover of $A_{m,n}$ with marked points $\hat P_i$, $\hat Q_j$, $i,j\in\Z$.
Consider any positive integral frieze $F$ on $A_{m,n}$ and define
\begin{equation}\label{eq mn} 
  u_{i,j}=\hat F(\hat P_i \hat Q_j),
\end{equation}
where $\hat F$ is the lift of $F$ to the universal cover constructed as in Example~\ref{ex-n}.

The numbers $u_{i,j}$ arranged as in Fig.~\ref{fig-u}(b) form a BiQF (the Ptolemy relations on the annulus guarantee that the quasi-diamond rule will hold in the grid). The constructed BiQF is $(m,n)$-periodic, as $\hat P_{i+m} \hat Q_{j+n}$  and $\hat P_i \hat Q_j$ are two lifts of the same arc on $A_{m,n}$.
  
Fig.~\ref{fig-mn} provides an example of a  $(3,2)$-periodic BiQF associated to a frieze on $A_{3,2}$.
  
\end{example}

\begin{theorem}
  \label{thm-mn}
There is a bijection between positive integral $(m,n)$-periodic BiQFs and positive integral friezes on the annulus $A_{m,n}$.
  
\end{theorem} 

\begin{proof}
  The proof is very similar to the one of Theorem~\ref{thm-n}. The map from friezes on $A_{m,n}$ to $(m,n)$-periodic BiQFs is constructed in Example~\ref{ex-mn}.
To construct the inverse map consider the elements $u_{i,j}, u_{i+1,j},\dots, u_{i+m,j}, u_{i+m,j+1}, \dots, u_{i+m,j+n-1}$ of an $(m,n)$-periodic BiQFs. These elements correspond to a triangulation $T$ of the annulus $A_{m,n}$ containing exactly two fans of sizes $m$ and $n$ and having no peripheral arcs. Assigning the arcs of $T$ with the listed  entries of the BiQF
 and computing other arcs we will get a positive integral frieze on $A_{m,n}$. As before, this provides an inverse map to the one described in  Example~\ref{ex-mn}.

\end{proof}

\section{$SL_2$-tilings with translational symmetries}
\label{sec-sl2-transl}
Consider a BiQF rotated by $45^\circ$ counterclockwise, we obtain an $SL_2$-tiling with the entries $u_{i,j}$, $i$ growing from the bottom to the top, $j$ growing from the left  to the right (e.g., rotating the BiQF shown in Fig.~\ref{fig-A22} one gets the $SL_2$-tiling shown in Fig.~\ref{u-sl}).

The results of Sections~\ref{sec n} and~\ref{sec mn} can be reformulated as follows: positive integral $SL_2$-tilings invariant under a translational symmetry $u_{i,j}=u_{i+m,j+n}$ with $mn>0$ are in bijection with friezes on annulus $A_{m,n}$. The aim of this section is to show that these exhaust all possible translational symmetries of positive integral $SL_2$-tilings, i.e. translational symmetry  $u_{i,j}=u_{i+m,j+n}$ with $mn\le 0$ never occurs. Our main tool for that will be Theorem~\ref{thm-sh}.

To show that there are no positive integral $(m,n)$-periodic $SL_2$-tilings with $mn<0$, we need the following lemma. By an \emph{arc} we mean a hyperbolic line (or an arc of a circle orthogonal to the boundary of the disc or half-plane). 

\begin{lemma}
\label{N}
  Let $\frac{a_1}{a_2},\frac{c_1}{c_2} \in \Q$ be two irreducible fractions. Let $N=N(\frac{a_1}{a_2},\frac{c_1}{c_2})$ be the number of edges of the Farey graph $\mathcal F$ intersected by the arc connecting $\frac{a_1}{a_2}$ with $\frac{c_1}{c_2}$. Then $|\det(\frac{a_1}{a_2},\frac{c_1}{c_2}))| \ge N$.

\end{lemma}  

\begin{proof}
We prove the statement by 
induction on the number of intersections $N$. When $N=0$ the statement is trivial. Given the statement for all arcs with less than $N$  intersections, we consider an arc $\gamma$ between $a_1/a_2$ and $c_1/c_2$ intersecting edges of $\mathcal F$ exactly $N$ times. Let  $\alpha\in \mathcal F$   be the last edge of $\mathcal F$ intersected by $\gamma$ (while going along $\gamma$ towards $c_1/c_2$), denote the endpoints of $\alpha$ by $b_1/b_2$ and $d_1/d_2$. In particular, the arcs connecting $c_1/c_2$ with $b_1/b_2$ and $d_1/d_2$ are edges of $\mathcal F$, see Fig.~\ref{fig-pf}(a). Notice also that both arcs connecting  $a_1/a_2$ with $b_1/b_2$ and $d_1/d_2$ intersect edges of $\mathcal F$ less than $N$ times, so we can apply the inductive assumption. Moreover, one of these arcs (in Fig.~\ref{fig-pf}(a), this is the one with endpoint $b_1/b_2$) intersects edges of $\mathcal F$ exactly $N-1$ times. Taking into account that $\alpha\in \mathcal F$ and hence $|\det(b_1/b_2,d_1/d_2)|=1$, and applying the Ptolemy relation, we obtain
$$
\begin{vmatrix} a_1&c_1\\a_2&c_2 \end{vmatrix}\cdot 1=
\begin{vmatrix} a_1&b_1\\a_2&b_2 \end{vmatrix}\cdot\begin{vmatrix} c_1&d_1\\c_2&d_2 \end{vmatrix}+
\begin{vmatrix} a_1&d_1\\a_2&d_2 \end{vmatrix}\cdot\begin{vmatrix} b_1&c_1\\b_2&c_2 \end{vmatrix}
\ge
(N-1) \cdot 1+1\cdot 1= N,
$$  as all terms  here are positive integers.

\end{proof}

\begin{figure}[!h]
 \epsfig{file=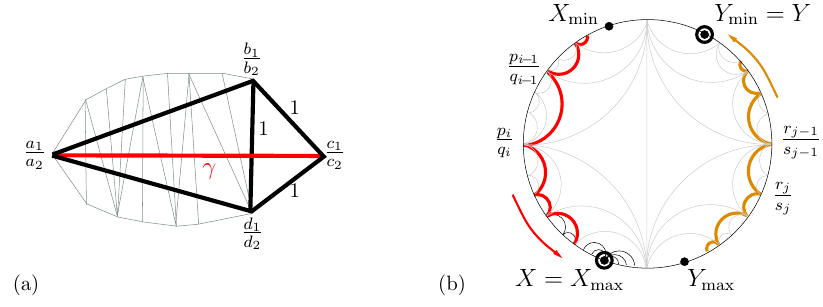,width=0.85\linewidth}  
\caption{(a) To the proof of Lemma~\ref{N};  (b) to the proof of Proposition~\ref{mn<0}.   }
\label{fig-pf}
\end{figure}

\begin{prop} 
\label{mn<0}  
There is no positive integral $SL_2$-tiling $\{u_{i,j}\}$ such that   $u_{i,j}=u_{i+m,j+n}$ for all $i,j\in \Z$, where  $mn< 0$.

\end{prop}

\begin{proof}
  Suppose that $\{u_{i,j}\}$ is a positive integral $SL_2$-tiling with  $u_{i,j}=u_{i+m,j+n}$ such that $mn< 0$.
We will assume $m>0$, $n<0$ (the case of $m<0$, $n>0$ is proved similarly). 
  
  Applying Theorem~\ref{thm-sh}(b), we  consider the corresponding two paths $(p_i/q_i)$ and $(r_j/s_j)$ on the Farey graph.  The paths do not intersect and, as in Remark~\ref{rem-paths},  the sequence $(p_i/q_i)$ follows the boundary of the disc counterclockwise while  $(r_j/s_j)$ clockwise. The condition  $u_{i,j}=u_{i+m,j+n}$  implies that   $ \begin{vmatrix} p_{km} & r_{kn} \\ q_{km} & s_{kn} \end{vmatrix}= \begin{vmatrix} p_0 & r_0 \\ q_0 & s_0 \end{vmatrix}=u_{0,0}$ for every positive $k\in\Z$. In particular, $u_{0,0}$ appears infinitely often amongst $u_{i,j}$ as $i\to +\infty$ and $j\to -\infty$. Our goal is to find a contradiction with this assumption.

Let $X_{\min}$ and $X_{\max}$ be the limit points of the path  $(p_i/q_i)$ as $i\to -\infty$ and $+\infty$ respectively, and   $Y_{\min}$ and $Y_{\max}$ be the limit points of $(r_j/s_j)$,  see Fig.~\ref{fig-pf}(b). By Theorem~\ref{thm-sh}, the limit points lie in $\overline \R$ in the following cyclic clockwise order $Y_{\min}, Y_{\max}, X_{\max}, X_{\min}$. Denote $Y=Y_{\min}=\lim\limits_{j\to -\infty}r_j/s_j$,  $X=X_{\max}=\lim\limits_{i\to +\infty}p_i/q_i$.

First, suppose that at least one of $X$ and $Y$, say $X$, is rational. Consider the set of edges of $\mathcal F$ incident to $X$, infinitely many of them are going to the left of $X$ and infinitely many to the right. This means that for every $N$ there is a number $i_0>0$ and $j_0<0$ such that  if $i>i_0$ and $j<j_0$ then the arc connecting $p_i/q_i$ with $r_j/s_j$ intersects more than $N$ edges of $\mathcal F$ (here we use that both sequences approach their limit points counterclockwise). Taking any $N>u_{0,0}$ and applying Lemma~\ref{N} we see that  $\begin{vmatrix} p_i & r_j \\ q_i & s_j \end{vmatrix}>u_{0,0}$ for all $i>i_0$ and  $j<j_0$ which contradicts the periodicity assumption.

Next, suppose that both $X$ and $Y$ are irrational. Then the arc $XY$ intersects infinitely many edges of $\mathcal F$.
Since $X$ and $Y$ are limit points of the considered sequences, for every edge $\gamma\in \mathcal F$ intersecting $XY$ there exist numbers $i_\gamma$ and $j_\gamma$ such that whenever  $i>i_\gamma$ and  $ j<j_\gamma$ the arc connecting $p_i/q_i$ with $r_j/s_j$ intersects $\gamma$. Again, by  Lemma~\ref{N} this implies that by taking any $N>u_{0,0}$ we can find $i_0,j_0$ such that  $ \begin{vmatrix} p_i & r_j \\ q_i & s_j \end{vmatrix}>u_{0,0}$ for all $i>i_0$ and $ j<j_0$.

\end{proof}

Similarly to Proposition~\ref{mn<0}, one can prove the following.

\begin{prop}
\label{m0}  
There is no positive integral $SL_2$-tiling $\{u_{i,j}\}$ such that   $u_{i,j}=u_{i+m,j}$ or  $u_{i,j}=u_{i,j+n}$.

\end{prop}  

\begin{cor}
\label{no 2}  
There is no positive integral $SL_2$-tiling with two linearly independent translational symmetries.

\end{cor}

\begin{proof}
An $SL_2$-tiling with two non-collinear translational symmetries would also contain a translation  $u_{i,j}=u_{i+m,j}$ for some $m>0$, which is impossible in view of~\ref{m0}. 

\end{proof}

Combining Propositions~\ref{mn<0} and~\ref{m0}, we obtain the following corollary.

\begin{cor}
\label{mn<=0}  
There is no positive integral $(m,n)$-periodic $SL_2$-tiling with $mn\le 0$.
\end{cor}

Now, combining results of Theorems~\ref{thm-n},~\ref{thm-mn} and Corollaries~\ref{no 2},~\ref{mn<=0}, we obtain the theorem.

\begin{theorem}
\label{thm-main}
There is a one-to-one correspondence between positive integral   $(m,n)$-periodic $SL_2$-tilings  and positive integral friezes on annulus $A_{m,n}$.
\end{theorem}

In view of unitarity of all friezes from annuli~\cite{GS}, Theorem~\ref{thm-main} can be reformulated as follows. 

\begin{cor}
\label{cor-bijection}  
There is a one-to-one correspondence between positive integral  $(m,n)$-periodic $SL_2$-tilings  and triangulations of the annulus $A_{m,n}$.

\end{cor}

\begin{remark}
  \label{rem-hyp}
  The correspondence obtained in Corollary~\ref{cor-bijection} can be also understood as follows.

 Consider a triangulation $T$ on the annulus $A_{m,n}$, and consider $T$ as the unitary triangulation of a frieze on $A_{m,n}$. Understanding the triangles of $T$ as ideal hyperbolic triangles with mutually tangent horocycles as in~\cite{FeTu}, we can lift the annulus to the hyperbolic plane so that its boundary will be lifted to two paths on the Farey graph. Notice that the paths do not intersect, and when applying the periodicity relation one path will be followed in clockwise direction and the other in a counterclockwise direction. By~\cite{Sh}, these two paths define a postive integral $SL_2$-tiling. This is precisely the $SL_2$-tiling from the bijective correspondence in Corollary~\ref{cor-bijection}. Moreover, there is a hyperbolic isometry of $\H^2$ taking both paths to themselves simultaneously (this isometry can be obtained by lifting to $\H^2$  the generator of the fundamental group of the annulus).     

\end{remark}

\begin{remark}
  \label{J}
  After the present paper was written, it was noticed to the authors by Peter Jorgensen that Corollary~\ref{cor-bijection} (and thus, in view of~\cite{GS}, Theorem~\ref{thm-main}) can be also deduced from a combination of results of~\cite{HJ} and~\cite{BHJ} as follows.

 Let $u_{i,j}$ be a positive integral $SL_2$-tiling. According to~\cite[Lemma 12.4]{BHJ}, if $u_{i,j}$ does not contain $1$, then it contains a unique minimal entry. Thus, if  $u_{i,j}$ has a translational symmetry, then there exists $(i_0,j_0)$ such that $u_{i_0,j_0}=1$. Moreover, by~\cite[Proposition 6.2]{HJ}, $u_{i_0+m,j_0+n}\ne 1$ for any $m,n$ with $mn<0$, and by ~\cite[Proposition 6.1]{HJ} the $i_0$-th row and $j_0$-th column may contain finitely many $1$'s only, which implies that in any $(m,n)$-periodic $SL_2$-tiling one has $mn>0$. Furthermore, if  $u_{i,j}$ is $(m,n)$-periodic, then  $u_{i_0+km,j_0+kn}=1$ for all $k\in\Z$, and then the main theorem of~\cite{HJ} states that such $SL_2$-tilings are in bijection with periodic triangulations of infinite strip, i.e. triangulations of annuli.

  \end{remark}

\section{Properties of $SL_2$-tilings with translational symmetries }
\label{sec-properties}

Theorem~\ref{thm-main} allows us to establish several properties of positive integral $SL_2$-tilings with translational symmetries (or, equivalently, periodic BiQFs), some of which are similar to the properties of frieze patterns described by Coxeter~\cite{C}.

\subsection{Relation to infinite frieze patterns}

An {\em infinite frieze pattern} is a table of numbers $v_{k,l}$ with $k,l\in\Z, l\ge k$, arranged as in Fig.~\ref{fig-inf} satisfying the following conditions. It has a row of zeroes (i.e. $v_{k,k}=0$), a row of $1$’s (i.e. $v_{k,k+1}=1$), and every small diamond satisfies the diamond rule: $v_{k,l}v_{k+1,l+1}-v_{k,l+1}v_{k+1,l}=1$. As for classical frieze patterns, the first non-trivial row $\{v_{k,k+2}\}$ is called the {\em quiddity sequence}, which completely defines the pattern.

\begin{figure}[!h]
  \epsfig{file=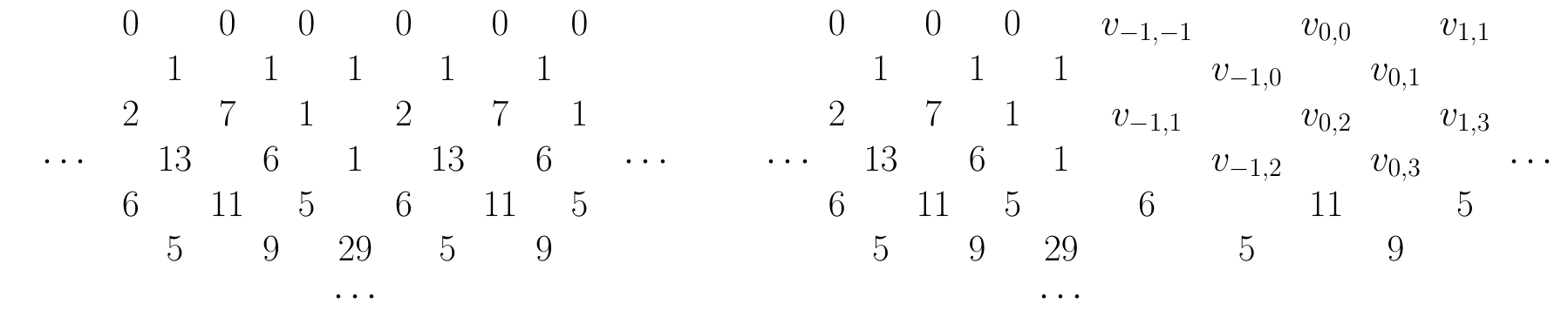,width=0.98\linewidth}  
\caption{Infinite frieze pattern: example (left) and indexing of its entries (right).}
\label{fig-inf}

\end{figure}

An infinite frieze pattern is \emph{$n$-periodic} if $v_{k,l}=v_{k+n,l+n}$ for every $k,l$. 
It follows from~\cite{BPT} that entries of every positive integral $n$-periodic  frieze pattern can be interpreted as follows: there exists a frieze $F$ on an annulus $A_{n,n'}$ (for some $n'$) or a punctured disc with $n$ boundary marked points, such that the entries are the values of $F$ on peripheral arcs. The quiddity sequence consists of the values of the frieze on short peripheral arcs (i.e., those bounding triangles with two sides at the boundary). Vice versa, for any frieze on annulus or punctured disc the values on its peripheral arcs form an infinite periodic frieze pattern. 

Given an $SL_2$-tiling with translational symmetry, we can construct the corresponding frieze on an annulus, and then the two infinite periodic frieze patterns corresponding to the two boundary components of this annulus (i.e. one is $m$-periodic and the other is $n$-periodic).
Their quiddity sequences can be found explicitly. Moreover, they can be read off from any column (respectively, row) of the  $SL_2$-tiling as follows.

\begin{prop}
\label{prop-quid} 
  Let $\{u_{i,j}\}$ be a positive integral $(m,n)$-periodic $SL_2$-tiling,  choose any $i_0,j_0\in\Z$. Define sequences
  $$a_i=\frac{u_{i-1,j_0}+u_{i+1,j_0}}{u_{i,j_0}},\qquad b_j=\frac{u_{i_0,j-1}+u_{i_0,j+1}}{u_{i_0,j}}.$$
Then sequences $(a_i),(b_j)$ are quiddity sequences of positive integer infinite periodic frieze patterns with periods $m$ and $n$ respectively. The patterns do not depend on $i_0,j_0$. 
    \end{prop}

    \begin{proof}
      Consider the frieze $F$ on $A_{m,n}$ constructed from the tiling  $\{u_{i,j}\}$. It is easy to see (Fig.~\ref{fig-periph}(b)) that $a_i$ is equal to the value of $\hat F$ on the short peripheral arc connecting $\hat P_{i-1}$ and $\hat P_{i+1}$ , where $\hat F$, as before, is the lift of $F$ to  the universal cover. Therefore, $a_i=a_{i+m}$, all $a_i$ do not depend on $j_0$, and the sequence $(a_i)$ defines an infinite periodic frieze pattern. Sequence $(b_j)$ can be treated similarly.
      
      \end{proof}

The following immediate corollary of Prop.~\ref{prop-quid} is a direct counterpart of~\cite[(8.2)]{C}

      \begin{cor}
        Let $\{u_{i,j}\}$ be a positive integral  $(m,n)$-periodic $SL_2$-tiling. Then $u_{i,j}$ divides $u_{i-1,j}+u_{i+1,j}$ and $u_{i,j-1}+u_{i,j+1}$ for any $i,j\in\Z$.

        Moreover, the ratio
        $\frac{u_{i,j-1}+u_{i,j+1}}{u_{i,j}}$ does not depend on $i$, and $\frac{u_{i-1,j}+u_{i+1,j}}{u_{i,j}}$ does not depend on $j$.
        \end{cor}

        Furthermore, any element of the two infinite periodic frieze patterns can be found explicitly from the tiling $\{u_{i,j}\}$.

        \begin{prop}
         \label{prop-ets} 
          Let $v_{k,l}$ be an $m$-periodic infinite frieze pattern corresponding to a positive integral  $(m,n)$-periodic $SL_2$-tiling   $\{u_{i,j}\}$. Then
          $$v_{k,l}=-\begin{vmatrix} u_{k,j} & u_{k,j+1} \\ u_{l,j} & u_{l,j+1} \end{vmatrix}$$
          for any $j\in\Z$.
          \end{prop}

          \begin{proof}
The proof follows from applying the Ptolemy relation to the quadrilateral with vertices $\hat P_k$, $\hat P_l$, $\hat Q_{j+1}$ and  $\hat Q_j$, see Fig.~\ref{fig-periph}(a). 

            \end{proof}

      \begin{example}
       For the BiQF shown in Fig.~\ref{fig-mn} (or for the corresponding $SL_2$-tiling), the two infinite frieze patterns have quiddities $(a_1,a_2,a_3)=(7,1,2)$ and $(b_1,b_2)=(2,3)$.  The first of these two infinite frieze patterns (the one with period $3$) is shown in Fig.~\ref{fig-inf}. 

      \end{example}

            Two periodic infinite frieze patterns constructed from different boundary components of a frieze on $A_{m,n}$ share a quantitative characteristic called {\em growth coefficient} (or simply {\em growth}). Namely, it was shown in~\cite{BFPT} that for an $n$-periodic infinite frieze pattern $v_{k,l}$ the value of $v_{i,i+n+1}-v_{i+1,i+n}$ does not depend on $i$.

            \begin{remark}
              \label{rm-growth}
              In terms of the corresponding frieze on $A_{m,n}$, the growth is precisely the trace of the hyperbolic isometry from Remark~\ref{rem-hyp}, which explains why the growth is the same for both  periodic infinite frieze patterns. This can be also seen as follows. The entry $v_{i,i+n+1}$ corresponds to a self-intersecting arc on the annulus.  Denote by $\Theta$ the evaluation of the element of the cluster algebra corresponding to the simple closed curve on the annulus (see Remark~\ref{rem-cl}). Smoothing the intersection (see~\cite{MSW}) leads to the equality $v_{i,i+n+1}=v_{i+1,i+n}+\Theta\cdot 1$, so the growth is equal to $\Theta$. On the other hand, it is well known that $\Theta$ is equal to the trace of the corresponding hyperbolic isometry. 
\end{remark}

              The growth can also be read off from the tiling $\{u_{i,j}\}$ as follows.

            \begin{prop}
              Let $\{u_{i,j}\}$ be a positive integral $(m,n)$-periodic $SL_2$-tiling. Then the value
              $$\frac{u_{i+m,j}+u_{i,j+n}}{u_{i,j}}$$
              is integer and does not depend on $i,j$. Moreover, it coincides with the growth of the two corresponding  periodic infinite frieze patterns. 
            \end{prop}

            \begin{proof}
             Due to Remark~\ref{rm-growth}, the growth of the two periodic infinite frieze patterns is equal to the evaluation $\Theta$ of the element of the cluster algebra corresponding to the simple closed curve $\gamma$ on the annulus. This element can be found by applying skein relation. Namely, $\gamma$ intersects any bridging arc, in particular the arc corresponding to the entry $u_{i,j}$ of the $SL_2$-tiling (for any $i,j\in\Z$). Smoothing the intersection leads to the equality $\Theta \cdot u_{i,j}= u_{i+m,j}+u_{i-m,j}$.  Since $u_{i-m,j}= u_{i,j+n}$, we obtain the required statement.

            \end{proof}

            \subsection{Is it integer?}

            We now want to address the following question. Given a (small) piece of table of positive integers, do they compose a part of an  $SL_2$-tiling with translational symmetry?

            Let $a<c$ and $b<d$ be integers. By a {\em lattice path} between $(a,b)$ and $(c,d)$ we mean a sequence of pairs  $\{(i_k,j_k)\}$ for $k=0,\dots,c-a+d-b$ with $(i_{0},j_{0})=(a,b)$, $(i_{c-a+d-b},j_{c-a+d-b})=(c,d)$, and $(i_{k+1},j_{k+1})$ equal to either $(i_{k},j_{k}+1)$ or $(i_{k}+1,j_{k})$. See Fig.~\ref{fig-path} for an example of a lattice path between $(i,j)$ and $(i+m,j+n)$.

\begin{figure}[!h]
  \epsfig{file=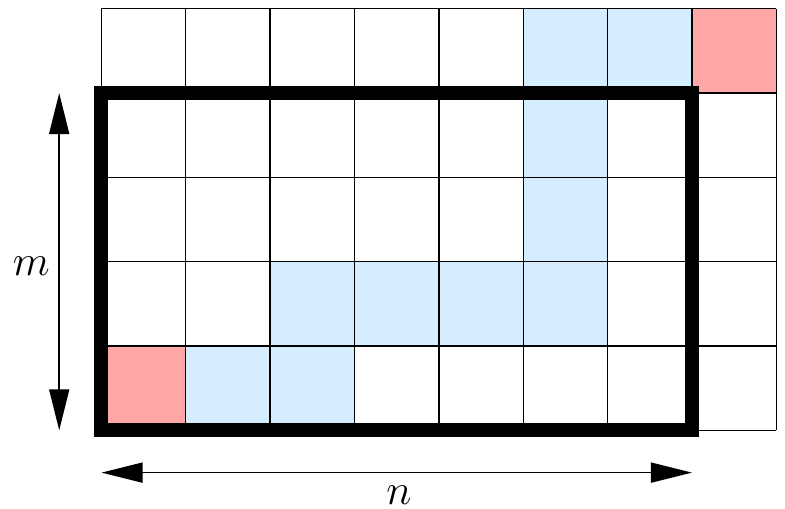,width=0.3\linewidth}  
\caption{A lattice path between $(i,j)$ and $(i+m,j+n)$.}
\label{fig-path}
\end{figure}

            \begin{prop}
            \label{zig-zag}  
              Let $m,n\in\Z_{>0}$, $i,j\in\Z$,  let $\Pi=\{(i_k,j_k)\}$ for $k=0,\dots, m+n$ be a lattice path between $(i,j)$ and $(i+m,j+n)$. Let $\{u_{i_k,j_k}\}$ be positive integers with  $u_{i,j}=u_{i+m,j+n}$, $(i_k,j_k)\in \Pi$.

              \begin{enumerate}
              \item
                If all  $u_{i_k,j_k}=1$ then $\{u_{i_k,j_k}\}$ extends uniquely to a  positive integral $(m,n)$-periodic $SL_2$-tiling  $\{u_{i,j}\}$.
              \item
                Assume that  $\{u_{i_k,j_k}\}$ satisfy the following two conditions:
                \begin{itemize}
                \item[-]
                  If $i_{k+1}-i_{k-1}$ is even, then $\frac{u_{i_{k+1},j_{k+1}}+u_{i_{k-1},j_{k-1}}}{u_{i_k,j_k}}\in\Z$.
                  \item[-] If $i_{k+1}-i_{k-1}$ is odd (i.e equal to $1$) then $\frac{u_{i_{k+1},j_{k+1}}\cdot \,u_{i_{k-1},j_{k-1}}+1}{u_{i_k,j_k}}\in\Z$.
                  \end{itemize}
 Then $\{u_{i_k,j_k}\}$ extends uniquely to a  positive integral $(m,n)$-periodic $SL_2$-tiling  $\{u_{i,j}\}$.

              \end{enumerate}

              \end{prop}
            
              \begin{proof}
                Consider an annulus $A_{m,n}$ with a triangulation $T_\Pi$ constructed as follows: an arc $P_rQ_s$ is contained in $T_\Pi$ if and only if $(r,s)$ is contained in $\Pi$. It is easy to see that  $T_\Pi$ is indeed a triangulation, and all its arcs are bridging. Moreover, the arcs with common index lie in the same fan of $T_\Pi$. The numbers  $u_{i_k,j_k}$ can be considered as the values of cluster variables from a seed of the cluster algebra corresponding to $T_\Pi$. In particular, since all $u_{i_k,j_k}$ are positive, all cluster variables are positive.

                If all  $u_{i_k,j_k}=1$, then $T_\Pi$ is a unitary triangulation, so we get a positive integral frieze on $A_{m,n}$, and the corresponding $SL_2$-tiling (see Thm.~\ref{thm-main}) contains the lattice path $\Pi$ filled in with $1$'s, which proves (1).

                Assume now that not all of  $u_{i_k,j_k}$ are equal to $1$. According to~\cite[Corollary 2.1]{BFZ}, for all cluster variables to be integer it is sufficient to check that one mutation in every direction preserves integrality, which is precisely the assumption of (2). 
                
                \end{proof}

\subsection{Units in $SL_2$-tilings with translational symmetries}

  The following proposition is a counterpart of the~\cite[(23)--(26)]{CC} (which lead to Conway--Coxeter classification of finite frieze patterns).           

                \begin{prop}
\label{prop-ears}
                  Let $\{u_{i,j}\}$ be a positive integral $(m,n)$-periodic $SL_2$-tiling. Then the following dichotomy holds.
                  \begin{itemize}
                  \item[(1)] Either there exists $(i_0,j_0)$ and a lattice path between $(i_0,j_0)$ and  $(i_0+m,j_0+n)$ such that $u_{i_k,j_k}=1$ for every $k=0,\dots,m+n$;
                  \item[(2)] or there exists $i_0$ (or $j_0$) such that the row $u_{i_0,j}$ is a sum of two rows  $u_{i_0-1,j}$ and $u_{i_0+1,j}$ (respectively, the column $u_{i,j_0}$ is a sum of two columns  $u_{i,j_0-1}$ and $u_{i,j_0+1}$). 
                    
                    \end{itemize}

                    Furthermore, in case (2) removing all rows   $u_{i_0+mt,j}$ (resp., all columns   $u_{i,j_0+nt}$) for all $t\in\Z$ leads to a $(m-1,n)$-periodic  $SL_2$-tiling (resp., $(m-1,n)$-periodic).
                  
                  \end{prop}
                  \begin{proof}
                    Consider the frieze on $A_{m,n}$ corresponding to the tiling $\{u_{i,j}\}$, and let $T$ be its unitary triangulation. If all arcs in $T$ are bridging, then the entries $\{u_{i,j}\}$ corresponding to arcs of $T$ form a lattice path between some $(i_0,j_0)$ and  $(i_0+m,j_0+n)$, so we get $(1)$.

                    Assume now that $T$ contains a peripheral arc $\gamma$.  Assume that the lift of $\gamma$ to the universal cover connects $\hat P_i$ with $\hat P_{i'}$. It is easy to see that $T$ then contains a peripheral arc whose lift connects $\hat P_{i_0-1}$ with $\hat P_{i_0+1}$ for some $i_0$. Given any $j\in \Z$, applying Ptolemy relation to the quadrilateral with vertices $\hat P_{i_0-1},\hat P_{i_0},\hat P_{i_0+1},\hat Q_{j}$ we obtain $u_{i_0,j}=u_{i_0-1,j}+u_{i_0+1,j}$ (cf. Fig.~\ref{fig-periph}(b)). Similarly, if the lift of $\gamma$ connects $\hat Q_j$ with $\hat Q_{j'}$, then there is $j_0$ such that  $u_{i,j_0}=u_{i,j_0-1}+u_{i,j_0+1}$ for any $i\in\Z$. Therefore, we get $(2)$, so we proved the dichotomy.

                    In particular, in case $(2)$, $T$ contains a triangle with two sides being boundary arcs. Removing this triangle (and thus all arcs incident to the removed marked point) leads to a frieze on $A_{m-1,n}$ or $A_{m,n-1}$, so we get  $SL_2$-tiling with required translational symmetry.

                    \end{proof}

                    \begin{remark}
The proof of Prop.~\ref{prop-ears} implies that case (2) corresponds to friezes with unitary triangulation containing a triangle with two sides on the boundary. Removing such triangles one by one, we obtain a frieze on an annulus with bridging unitary triangulation. Equivalently, for any positive integral $SL_2$-tiling with translational symmetry one can remove certain number of rows and columns to obtain a tiling with a lattice path of $1$'s.   
                      
                      \end{remark}


\begin{thebibliography}{MGOST14}
\bibitem[ARS10]{ARS} I.~Assem, C.~Reutenauer, D.~Smith, {\em Friezes}, Adv. Math. 225 (2010), 3134--3165.

  
\bibitem[BFPT19]{BFPT}  K.~Baur, K.~Fellner, M.~J.~Parsons, M.~Tschabold, {\em Growth behaviour of periodic tame friezes}, Rev. Mat. Iberoam. 35 (2019), 575--606.

  
  \bibitem[BPT16]{BPT}  K.~Baur, M.~J.~Parsons, M.~Tschabold, {\em Infinite friezes}, European J. Combin. 54 (2016), 220--237. 

\bibitem[BR10]{BR}   F.~Bergeron, C.~Reutenauer, {\em $SL_k$-tilings of the plane}, Illinois J. Math., 54 (2010), 263--300.


\bibitem[BHJ17]{BHJ} C.~Bessenrodt, T.~Holm, P.~Jorgensen, \emph{All $SL_2$-tilings come from infinite triangulations}, Adv. Math. 315 (2017), 194--245.
  

\bibitem[CC73]{CC}
 J.~H.~Conway, H.~S.~M.~Coxeter, {\em Triangulated polygons and frieze patterns}, Math. Gaz. 57 (1973) 87--94, 175--183.


\bibitem[Cox71]{C} 
H.~S.~M. Coxeter, {\em Frieze patterns}, Acta Arith. 18 (1971), 297--310.

\bibitem[DP15]{DP}
G.~Dupont, F.~Palesi, {\em Quasi-cluster algebras from non-orientable
surfaces}, Journal of Algebraic Combinatorics, 42(2):429--472, 2015.


\bibitem[FT24]{FeTu} A.~Felikson, P.~Tumarkin, {\em   Friezes from surfaces and Farey triangulation}, arXiv:2410.13511

\bibitem[FST08]{FST} S.~Fomin, M.~Shapiro, D.~Thurston, {\em Cluster algebras and triangulated surfaces. Part {\rm I}: Cluster complexes}, Acta Math. 201 (2008), 83--146.
  
\bibitem[FT18]{FT} S.~Fomin, D.~Thurston, {\em Cluster algebras and triangulated surfaces. Part {\rm II}: Lambda lengths}, Mem. Amer. Math. Soc. 255 (2018), no. 1223, v+97 pp.
  
\bibitem[BFZ05]{BFZ} A.~Berenstein, S.~Fomin, A.~Zelevinsky, \emph{Cluster algebras. III. Upper bounds and double Bruhat cells}, Duke Math. J. 126 (2005), 1--52.

\bibitem[GS20]{GS} E.~Gunawan, R.~Schiffler, {\em Frieze vectors and unitary friezes}, J. Comb. 11 (2020), 681--703.

\bibitem[HJ13]{HJ}  
T.~Holm, P.~Jorgensen, {\em $SL_2$-tilings and triangulations of the strip}, J.
Combin. Theory Ser. A, 120(7):1817--1834, 2013.

\bibitem[MGO19]{MGO}
  S.~Morier-Genoud, V.~Ovsienko, {\em Farey boat I. Continued fractions and triangulations, modular group and polygon dissections.},  Jahresber. Dtsch. Math.-Ver. 121 (2) (2019), 91--136.

\bibitem[MGOT15]{MGOT} 
 S.~Morier-Genoud, V.~Ovsienko, and S.~Tabachnikov, {\em $SL_2(\mathbb Z)$-tilings of the torus, Coxeter-Conway friezes
and Farey triangulations}, Enseign. Math. 61 (2015), 71--92.
  

\bibitem[MSW11]{MSW} G.~Musiker, R.~Schiffler, L.~Williams, {\em Positivity for cluster algebras from surfaces}, Adv. Math. 227 (2011), 2241--2308.

\bibitem[Pen87]{P} R.~C.~Penner, {\em The decorated Teichm\"uller space of punctured surfaces},
Comm. Math. Phys. 113 (1987),  299--339.

  
\bibitem[Sh23]{Sh} I.~Short, {\em Classifying $SL_2$-tilings}, Trans. Amer. Math. Soc. 376 (2023), 1--38.

\end{thebibliography}
\end{document}